\documentclass[a4paper]{amsart}
\numberwithin{equation}{section}
\newtheorem{theorem}{Theorem}[section]

\newtheorem{lemma}{Lemma}[section]
\newtheorem{example}{Example}[section]

\theoremstyle{remark}
\newtheorem{remark}{Remark}[section]
\usepackage{hyperref}
\usepackage{color}
\usepackage{graphicx}
\usepackage{subfigure}
\title[Starlike log--harmonic functions]
 {Notes on the starlike log--harmonic mappings of order alpha}
\begin{document}
\begin{abstract}
Let $h$ and $g$ be two analytic functions in the unit disc $\Delta$ that $g(0)=1$. Also let $\beta$ be a complex number with ${\rm Re}\{\beta\}>-1/2$. A function $f$ is said to be log--harmonic mapping if it has the following representation
\begin{equation*}
  f(z)=z |z|^{2\beta} h(z)\overline{g(z)}\quad (z\in \Delta).
\end{equation*}
A log--harmonic mapping $f$ is said to be starlike log--harmonic mapping of order $\alpha$, where $0\leq \alpha<1$, if
\begin{equation*}
  {\rm Re}\left\{\frac{zf_z -\overline{z}f_{\overline{z}}}{f}\right\}>\alpha\quad(z\in \Delta).
\end{equation*}
In this paper, by use of the subordination principle, we study some geometric properties of the starlike log--harmonic mappings of order $\alpha$. Also, we estimate the Jacobian of log--harmonic mappings.
\end{abstract}

\author[R. Kargar and H. Mahzoon]
       {R. Kargar and H. Mahzoon}
\address{Young Researchers and Elite Club,
Ardabil Branch, Islamic Azad University, Ardabil, Iran}
       \email {rkargar1983@gmail.com}
\address{Department of Mathematics, Firoozkouh
Branch, Islamic Azad University, Firoozkouh, Iran} \email {hesammahzoon1@gmail.com}
\maketitle
\section{Introduction}
Let $\mathcal{H}(\Delta)$ be the family of all analytic functions in the open unit disc $\Delta:=\{z : |z|<1\}$. Let $f$ and $g$ be two members of the class $\mathcal{H}(\Delta)$ which satisfy the normalized conditions $f(0)=0=f'(0)-1$ and $g(0)=0=g'(0)-1$. We say that a function $f$ is subordinate to $g$, written $f (z)\prec g(z)$ or $f\prec g$, if there exists an analytic function $\psi$, known as a Schwarz function, with $\psi(0)=0$ and $|\psi(z)|\leq|z|$ such that $f(z)=g(\psi(z))$ for all $z\in\Delta$. The set of all univalent (one--to--one) functions $f$ in $\Delta$ is denoted by $\mathcal{U}$.
Furthermore, if $g\in\mathcal{U}$ in $\Delta$, then we have the following equivalence relation
\begin{equation*}
    f(z)\prec g(z) \Leftrightarrow f(0)=g(0)\quad {\rm and}\quad f (\Delta)\subset g(\Delta).
\end{equation*}
 Also let $\mathcal{B}(\Delta)$ denote the set of all functions $w\in\mathcal{H}(\Delta)$ satisfying $|w(z)|<1$ in $\Delta$. A mapping $f$ is said to be log--harmonic in $\Delta$, if there exists an
analytic function $w\in\mathcal{B}(\Delta)$ such that $f$ is a solution of the nonlinear elliptic partial
differential equation
\begin{equation}\label{equ. log}
  \frac{\overline{f_{\overline{z}}}}{\overline{f}}=w(z)\frac{f_z}{f},
\end{equation}
where $w$ is the second complex dilatation of $f$ and $w\in \mathcal{B}(\Delta)$.
We note that if $f$ is a non--vanishing log--harmonic mapping, then
$f$ has the following form
\begin{equation*}
f(z)=h(z)\overline{g(z)},
\end{equation*}
 where both $h$ and $g$ are  non--vanishing analytic functions, and that if $f$ vanishes at $z=0$ but is not identically zero, then $f$ has the following form
\begin{equation}\label{repres}
  f(z)=z |z|^{2\beta} h(z)\overline{g(z)},
\end{equation}
where ${\rm Re}\{\beta\}>-1/2$, $h$ and $g$ are in $\mathcal{H}(\Delta)$, $h(0)\neq0$ and $g(0)=1$ (see \cite{Ab.Alhadi Trans.}). The class of functions of this
form has been studied extensively by many works, see for example \cite{ZAb. 1996}-- \cite{ali2016} and \cite{Mao}.

We continue the discussion with an example.
\begin{example}
  If $f(1)=1$ and ${\rm Re}\{\beta\}>-1/2$, then the function
  \begin{equation*}
  f_\beta(z)=z |z|^{2\beta}\quad(z\in\Delta),
  \end{equation*}
   is a solution of the equation \eqref{equ. log} in the complex plane $\mathbb{C}$ with $w\equiv \overline{\beta}/(1+\beta)$. It is a simple exercise that $f_\beta$ maps the unit disc $\Delta$ onto itself. The Figure \ref{Fig1} shows the image of the unit disc $\Delta$ under the function $f_\beta$ in two different cases.
\end{example}
\begin{figure}[!ht]
\centering
\subfigure[]{
\includegraphics[width=5cm]{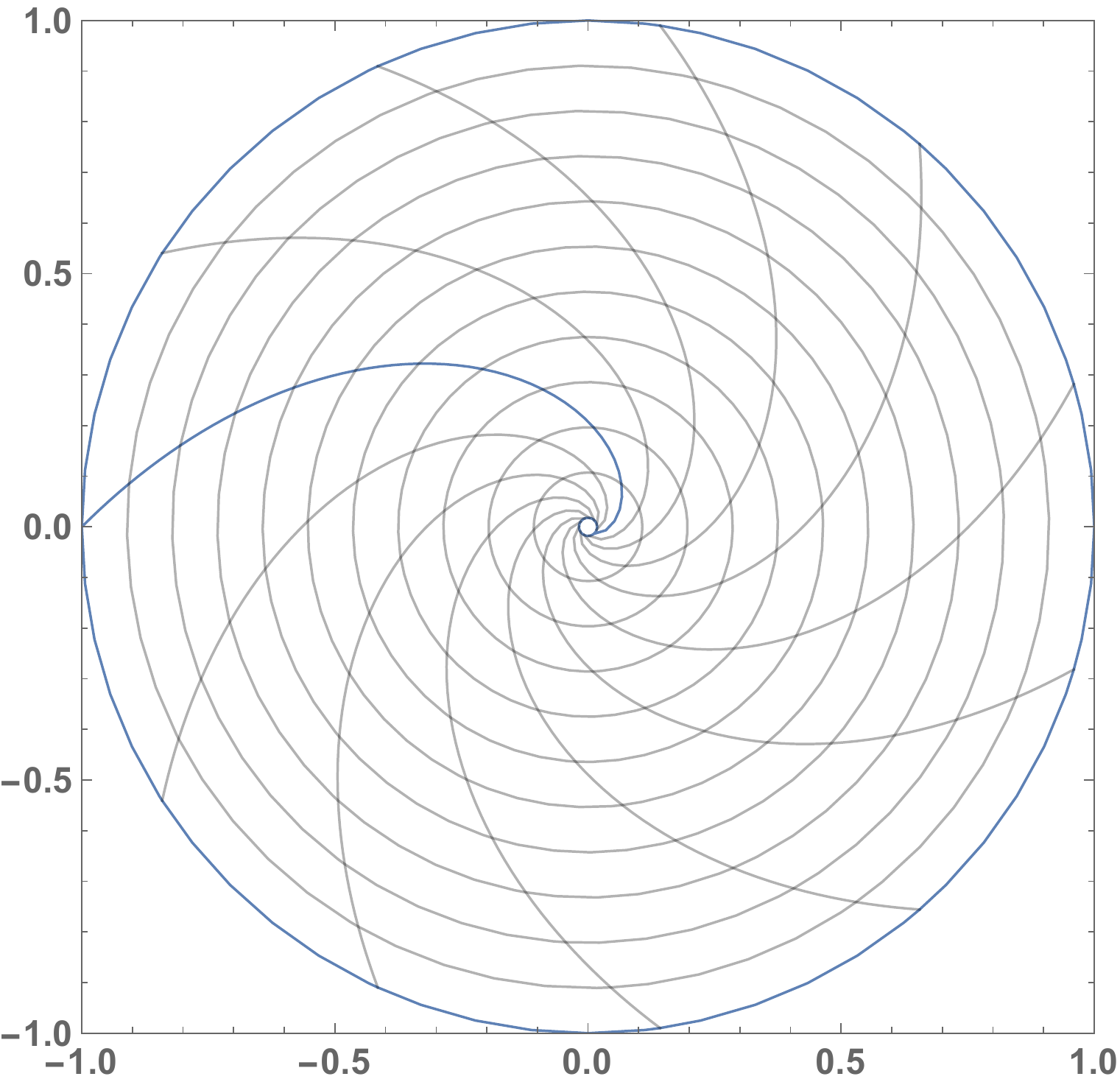}
	\label{fig:subfig1}
}
\hspace*{10mm}
\subfigure[]{
\includegraphics[width=5cm]{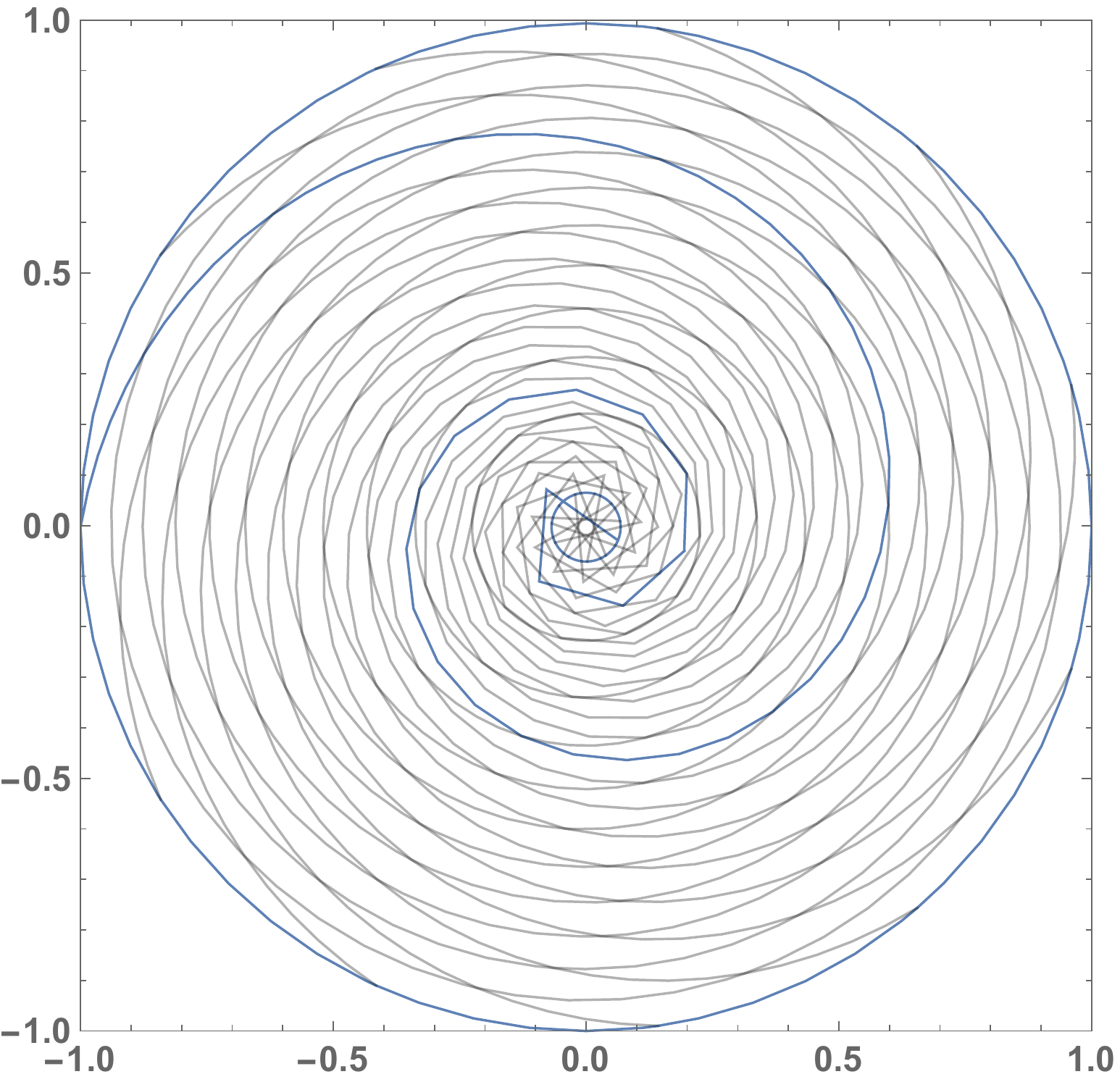}
	\label{fig:subfig2}
}

\caption[The boundary curve of $f_{i}(\Delta)$]
{\subref{fig:subfig1}: The boundary curve of $f_{i}(\Delta)$  ،
 \subref{fig:subfig2}: The boundary curve of $f_{-1/3+4i}(\Delta)$،
 }
\label{fig:subfig01}
\label{Fig1}
\end{figure}

We denote by $J_f(z)$ the Jacobian of log--harmonic mappings $f$ as follows
\begin{equation}\label{jaco}
  J_f(z)=|f_z|^2\left(1-|w(z)|^2\right).
\end{equation}
Since $w\in\mathcal{B}(\Delta)$, thus $J_f(z)>0$ and all non--constant log--harmonic mappings are sense--preserving and
open in the disc $\Delta$.

Let $f(z)=z |z|^{2\beta} h(z)\overline{g(z)}$, where $h(0)=g(0)=1$, be a log--harmonic mapping. Then we say that $f$ is a
{\it starlike log--harmonic ($\mathcal{SLH}$) mapping of order} $\alpha$, where $0\leq \alpha<1$, if
\begin{equation}\label{SLH}
  \frac{\partial}{\partial\theta}\arg f\left(re^{i\theta}\right)={\rm Re}\left\{\frac{zf_z -\overline{z}f_{\overline{z}}}{f}\right\}>\alpha\quad(z\in \Delta).
\end{equation}
We denote by $\mathcal{S^*_{LH}}(\alpha)$, the set of all starlike log--harmonic mappings
of order $\alpha$.

The following lemma will be useful.
\begin{lemma}\label{MIMO}
\cite[p. 35]{MM-book} Let $\Xi$ be a simply connected domain in the complex plane
$\mathbb{C}$, $\Xi\neq\mathbb{C}$ and let $b$ be a complex number with
${\rm Re}\{b\}>0$. Suppose that a function $
\psi:\mathbb{C}^2\times \Delta\rightarrow \mathbb{C}$ satisfies the
condition
\begin{equation*}
   \psi(i\rho,\sigma;z)\not\in
   \Xi
\end{equation*}
for all real $\rho,\sigma\leq-\mid
b-i\rho\mid^2/(2{\rm Re}\{b\})$ and all $z\in\Delta$. If the
function $p(z)$ defined by $p(z)=b+b_1z+b_2z^2+\cdots$ is analytic
in $\Delta$ and if
\begin{equation*}
   \psi(p(z),zp'(z);z)\in \Xi,
\end{equation*}
then ${\rm Re} \{p(z)\}>0$ in $\Delta$.
\end{lemma}

One of the goals of this paper is to investigate some geometric properties of the starlike log--harmonic mappings of order $\alpha$, another is to give an estimate for the Jacobian of log--harmonic mappings of the form \eqref{repres}.
\section{Main Results}

In the first result, by use of the subordination principle, we give a necessary and sufficient condition for functions of the form \eqref{repres} belonging to the class $\mathcal{S^*_{LH}}(\alpha)$.

\begin{theorem}\label{th. sub}
  Let $\alpha\in[0,1)$ and ${\rm Re}\{\beta\}>-1/2$. Then the function $f(z)=z |z|^{2\beta} h(z)\overline{g(z)}$ belongs to the class $\mathcal{S^*_{LH}}(\alpha)$ if, and only if,
  \begin{equation}\label{eq. sub th.}
    \left(z\frac{h'(z)}{h(z)}-z\frac{g'(z)}{g(z)}\right)\prec \frac{2(1-\alpha)z}{1-z}\quad(z\in \Delta).
  \end{equation}
\end{theorem}
\begin{proof}
  Let the function $f(z)=z |z|^{2\beta} h(z)\overline{g(z)}$ belongs to the class $\mathcal{S^*_{LH}}(\alpha)$ where $0\leq\alpha<1$. Then by a simple check we get
  \begin{equation}\label{zfz}
    \frac{zf_z}{f}=1+\beta+z\frac{h'(z)}{h(z)}\quad(z\in \Delta)
  \end{equation}
  and
  \begin{equation}\label{zbar fzbar}
     \frac{\overline{z}f_{\overline{z}}}{f}
     =\beta+\overline{z}\overline{\left(\frac{g'(z)}{g(z)}\right)}\quad(z\in \Delta).
  \end{equation}
  Thus
  \begin{equation}\label{zf pr. th.1}
    \frac{zf_z -\overline{z}f_{\overline{z}}}{f}
    =1+z\frac{h'(z)}{h(z)}-\overline{z}\overline{\left(\frac{g'(z)}{g(z)}\right)}
    \quad(z\in \Delta).
  \end{equation}
  Then $f\in \mathcal{S^*_{LH}}(\alpha)$ if, and only if,
  \begin{align*}
    \alpha< {\rm Re}\left\{\frac{zf_z -\overline{z}f_{\overline{z}}}{f}\right\}&={\rm Re}\left\{1+z\frac{h'(z)}{h(z)}
    -\overline{z}\overline{\left(\frac{g'(z)}{g(z)}\right)}\right\}\\ &={\rm Re}\left\{1+z\frac{h'(z)}{h(z)}
    -z\frac{g'(z)}{g(z)}\right\}.
    \end{align*}
    Therefore
\begin{equation*}
{\rm Re}\left\{z\frac{h'(z)}{h(z)}-z\frac{g'(z)}{g(z)}\right\}>\alpha-1\quad(z\in \Delta).
\end{equation*}
    Now by use of the subordination principle we get
\begin{equation*}
     {\rm Re}\left\{ z\frac{h'(z)}{h(z)}-z\frac{g'(z)}{g(z)}\right\}\prec
     \frac{2(1-\alpha)z}{1-z}\quad(z\in \Delta).
\end{equation*}
This ends the proof.
\end{proof}
We now have the following lemma directly.
\begin{lemma}\label{lemma re part}
  The function $f(z)=z |z|^{2\beta} h(z)\overline{g(z)}$ {\rm(}${\rm Re}\{\beta\}>-1/2${\rm)} belongs to the class $\mathcal{S^*_{LH}}(\alpha)$ if, and only if,
  \begin{equation}\label{Real part}
    1+{\rm Re}\left\{z\frac{h'(z)}{h(z)}-z\frac{g'(z)}{g(z)}\right\}>\alpha\quad(0\leq \alpha<1, z\in\Delta).
  \end{equation}
\end{lemma}
Following, we give the representation theorem for the function $h$ of the mappings $f$ of the form \eqref{repres} in the set $\mathcal{S^*_{LH}}(\alpha)$.

\begin{theorem}
 A function $f(z)=z |z|^{2\beta} h(z)\overline{g(z)}$ belongs to the class $\mathcal{S^*_{LH}}(\alpha)$ if, and only if,
  \begin{equation}\label{repres.}
    h(z)=g(z)\exp\left(\int_{0}^{z}\frac{2(1-\alpha)\psi(t)}{t(1-\psi(t))}{\rm d}t\right)\quad(z\in\Delta),
  \end{equation}
  where $\psi$ is a Schwarz function, $0\leq \alpha<1$ and ${\rm Re}\{\beta\}>-1/2$.
\end{theorem}
\begin{proof}
  If $f(z)=z |z|^{2\beta} h(z)\overline{g(z)}$ belongs to the class $\mathcal{S^*_{LH}}(\alpha)$ where $0\leq \alpha<1$, then by Theorem \ref{th. sub} and by definition of subordination, there exists a Schwarz function $\psi$ such that
\begin{equation*}
    \left(z\frac{h'(z)}{h(z)}-z\frac{g'(z)}{g(z)}\right)= \frac{2(1-\alpha)\psi(z)}{1-\psi(z)}\quad(z\in \Delta),
\end{equation*}
  or equivalently
\begin{equation}\label{log hg}
  \left\{\log \frac{h(z)}{g(z)}\right\}'=\frac{2(1-\alpha)\psi(z)}{z(1-\psi(z))}\quad(z\in \Delta).
\end{equation}
  Integrating the last equality \eqref{log hg}, we get \eqref{repres.}. On the other hand, it is an easy calculation that a function having of the form \eqref{repres.} satisfies the condition \eqref{eq. sub th.}.
\end{proof}

Applying formula \eqref{repres.} for $\psi(z)=z$ gives the following.
\begin{example}
Let $g(z)$ be an analytic function with $g(0)=1$. Then the function
\begin{equation}\label{f(z)=}
  F_{\alpha,\beta}(z)=z|z|^{2\beta}\frac{|g(z)|^2}{(1-z)^{2(1-\alpha)}}\quad(z\in\Delta, 0\leq \alpha<1),
\end{equation}
is a starlike log--harmonic mapping of order $\alpha$.
\end{example}
\begin{remark}
  Since $g(0)=1$, thus if we consider the constant function $g(z)=1$ in \eqref{f(z)=} and let $\beta=0$, then we obtain
  \begin{equation*}
    f(z)=\frac{z}{(1-z)^{2(1-\alpha)}}\quad(0\leq \alpha<1).
  \end{equation*}
  Moreover, the Koebe function of order $\alpha$ (in particular, the Koebe function) belongs to the class $\mathcal{S^*_{LH}}(\alpha)$.
\end{remark}
\begin{example}
  Let $\psi(z)=z/(1+z)$ and $|z|<0.65$. Then it is easy to see that $\psi$ is Schwarz function. By putting $\psi(z)=z/(1+z)$ in the equation \eqref{repres.}, we get
  \begin{equation*}
  h(z)=g(z) e^{2(1-\alpha)z}\quad(|z|<0.65).
  \end{equation*}
  Thus the function
\begin{equation*}
  f_{\alpha,\beta}(z)=z|z|^{2\beta}e^{2(1-\alpha)z}|g(z)|^2\quad({\rm Re}\{\beta\}>-1/2, |z|<0.65),
\end{equation*}
  belongs to the class $ST_{LH}(\alpha)$, where $g$ is analytic with $g(0)=1$. In particular, the function
  \begin{equation*}
  \widehat{f}_{\alpha,\beta}(z)=z|z|^{2\beta}e^{2(1-\alpha)z}\quad(|z|<0.65),
  \end{equation*}
   is a starlike log--harmonic mapping of order $\alpha$ where ${\rm Re}\{\beta\}>-1/2$. Also, $\widehat{f}_{0,0}(z)=ze^{2z}$ $(|z|<0.65)$ is $\mathcal{SLH}$ mapping. It is a simple exercise that the radius of injectivity of the function $\widehat{f}_{0,0}$ is $1/2$. The Figures \ref{fig:subfig1} and \ref{fig:subfig2} show the image of the disc $|z|<0.65$ and $|z|<1/2$ under the function $\widehat{f}_{0,0}$, respectively. 
\end{example}
\begin{figure}[!ht]
\centering
\subfigure[]{
\includegraphics[width=5cm]{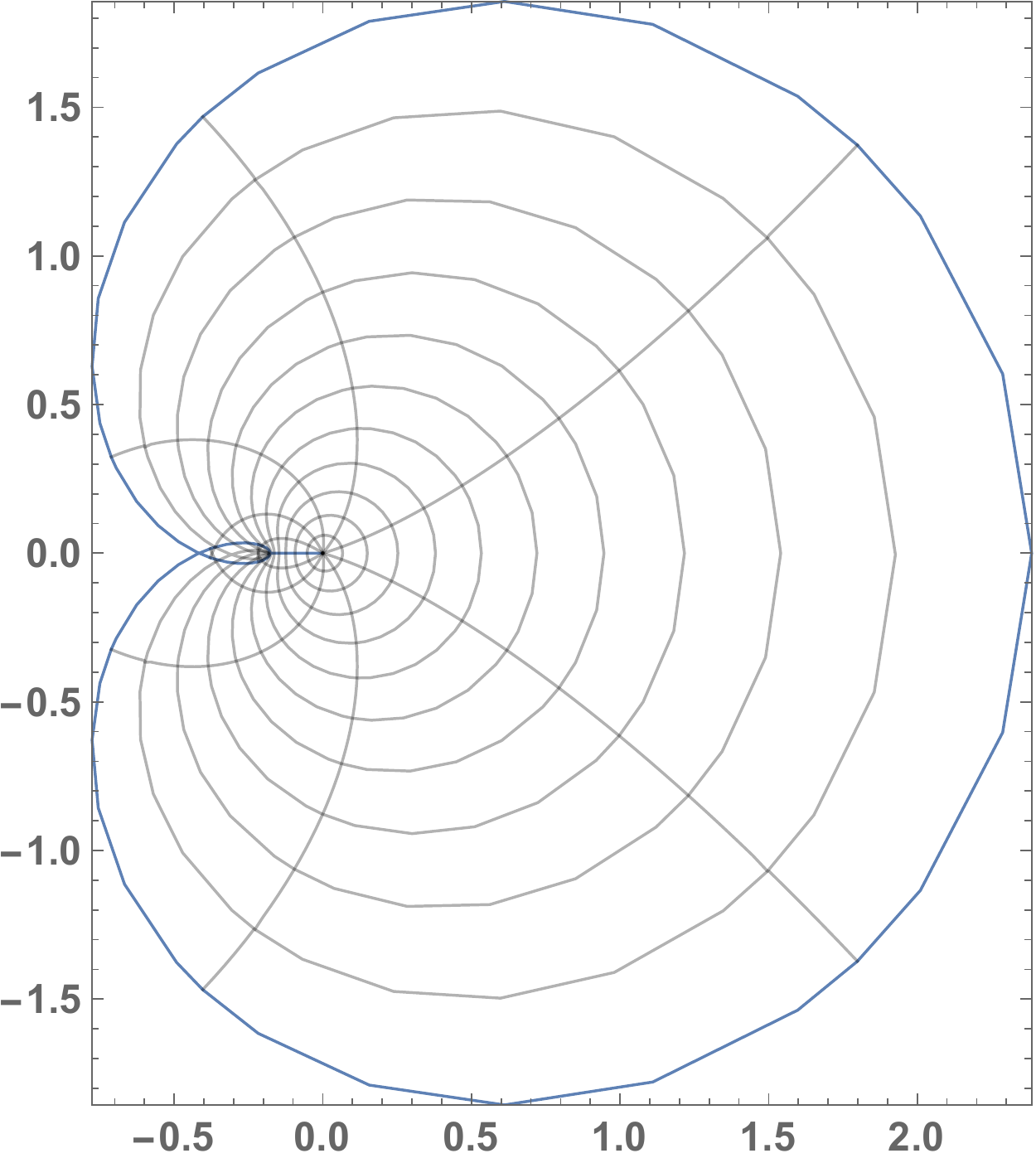}
	\label{fig:subfig1}
}
\hspace*{10mm}
\subfigure[]{
\includegraphics[width=5cm]{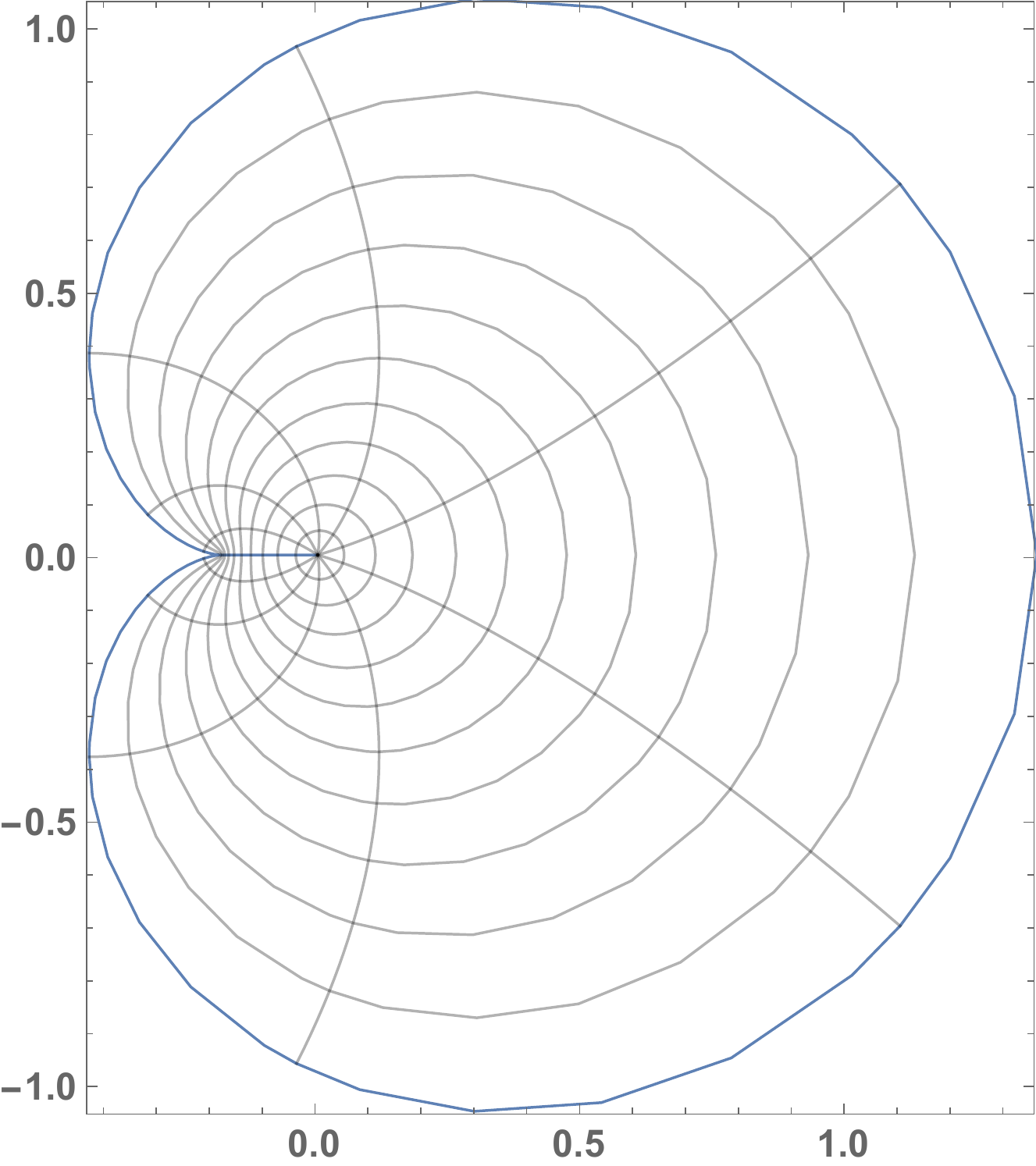}
	\label{fig:subfig2}
}
\caption[The boundary curve of $f(\Delta)$]
{\subref{fig:subfig1}: The boundary curve of $\widehat{f}_{0,0}(|z|<0.65)$  ،
 \subref{fig:subfig2}: The boundary curve of $\widehat{f}_{0,0}(|z|<1/2)$،
 }
\label{fig:subfig01}
\label{Fig2}
\end{figure}
\begin{theorem}
Let $\alpha\in[0,1)$ and ${\rm Re}\{\beta\}>-1/2$.
  Let a function $f(z)=z |z|^{2\beta} h(z)\overline{g(z)}$ belongs to the class $\mathcal{S^*_{LH}}(\alpha)$ where
  \begin{equation*}
  h(z)=1+\sum_{n=1}^{\infty}a_n z^n \ \ { and}\ \ g(z)=1+\sum_{n=1}^{\infty}b_n z^n,
  \end{equation*}
  be a log--harmonic mapping, and let
  \begin{equation*}
  \frac{zh'(z)}{h(z)}=\sum_{n=1}^{\infty}h_n z^n\quad {\rm and}\quad \frac{zg'(z)}{g(z)}=\sum_{n=1}^{\infty}g_n z^n.
\end{equation*}
If
  \begin{equation}\label{sum hn beta<1}
    \sum_{n=1}^{\infty}|h_n|<1-|\beta|
  \end{equation}
and
  \begin{equation}\label{sum hn+gn<}
    \sum_{n=1}^{\infty}\left(|h_n|+|g_n|\right)\leq 1-2|\beta|,
  \end{equation}
  then $f$ is sense preserving.
\end{theorem}
\begin{proof}
To show that $f$ is sense preserving, we need to prove
that $|w(z)|<1$ where $w$ denotes the dilatation of $f$. A simple calculation using \eqref{equ. log} and \eqref{repres}, gives us
  \begin{equation}\label{w(z)}
  w(z)=\frac{\overline{f_{\overline{z}}}}{\overline{f}}\frac{f}{f_z}
  =\frac{\overline{\beta}+z\frac{g'(z)}{g(z)}}
  {1+\beta+ z\frac{h'(z)}{h(z)}}\quad(z\in\Delta).
\end{equation}
Clearly $w(0)=\overline{\beta}/(1+\beta)=:\gamma$, where $\beta\in \mathbb{C}$ and ${\rm Re}\{\beta\}>-1/2$. We have
\begin{align*}
  |w(z)| &= \left|\frac{\overline{\beta}+z\frac{g'(z)}{g(z)}}
  {1+\beta+ z\frac{h'(z)}{h(z)}}\right| \\
  &=\left|\frac{\overline{\beta}+\sum_{n=1}^{\infty}g_n z^n}{1+\beta+\sum_{n=1}^{\infty}h_n z^n}\right|\\
  &\leq \frac{|\overline{\beta}|+\sum_{n=1}^{\infty}|g_n| |z|^n}{1-|\beta|-\sum_{n=1}^{\infty}|h_n| |z|^n}\\
  &< \frac{|\overline{\beta}|+\sum_{n=1}^{\infty}|g_n|}{1-|\beta|-\sum_{n=1}^{\infty}|h_n|}
  \leq1.
\end{align*}
This proves the theorem.
\end{proof}
To prove the following theorem, we will use the method of \cite[Theorem 2.1]{kessiberian}.
\begin{theorem}
  Let $\alpha\in(1/2,1)$. If the function $f(z)=z |z|^{2\beta} h(z)\overline{g(z)}$ belongs to the class $\mathcal{S^*_{LH}}(\alpha)$, then
  \begin{equation}\label{Re h/g}
    {\rm Re}\left\{\frac{h(z)}{g(z)}\right\}>\mu(\alpha):=\frac{1}{3-2\alpha}\quad(z\in \Delta).
  \end{equation}
\end{theorem}
\begin{proof}
  For convenience, we put $\mu(\alpha):=\mu$ and note that $\mu\in(1/2,1)$ when $ \alpha\in(1/2,1)$. Define
  \begin{equation}\label{p(z)}
    p(z)=\frac{1}{1-\mu}\left(\frac{h(z)}{g(z)}-\mu\right)\quad(z\in\Delta).
  \end{equation}
  Then $p(z)$ is analytic function in $\Delta$ and $p(0)=1$. A simple check gives
  \begin{equation*}
    1+z\frac{h'(z)}{h(z)}-z\frac{g'(z)}{g(z)}=1+\frac{(1-\mu)zp'(z)}{\mu+(1-\mu)p(z)}
    =:\phi(p(z),zp'(z)),
  \end{equation*}
  where
  \begin{equation}\label{phi}
    \phi(x,y)=1+\frac{(1-\mu)y}{\mu+(1-\mu)x}.
  \end{equation}
  By Lemma \ref{lemma re part}, since $f(z)=z |z|^{2\beta} h(z)\overline{g(z)}\in \mathcal{S^*_{LH}}(\alpha)$, we can consider
  \begin{equation*}
    \{\phi(p(z),zp'(z)):z\in \Delta\}\subset \{w\in \mathbb{C}:{\rm Re}\{w\}>\alpha\}=:\Omega_\alpha.
  \end{equation*}
  Now for all real $\rho$ and $\sigma\leq -\frac{1}{2}(1+\rho^2)$, we get
  \begin{align*}
    {\rm Re}\{\phi(i\rho,\sigma)\} &={\rm Re} \left\{1+\frac{(1-\mu)\sigma}{\mu+(1-\mu)i\rho}\right\}\\
    &=1+\frac{\mu(1-\mu)\sigma}{\mu^2+(1-\mu)^2\rho^2}\\
    &\leq 1-\frac{1}{2}\mu(1-\mu)Q(\rho),
  \end{align*}
  where
  \begin{equation}\label{Q(rho)}
    Q(\rho):=\frac{1+\rho^2}{\mu^2+(1-\mu)^2\rho^2}.
  \end{equation}
  It is easy to see that
  \begin{equation*}
    Q'(\rho)=\frac{2(2\mu-1)\rho}{\left(\mu^2+(1-\mu)^2\rho^2\right)^2}
  \end{equation*}
  and $Q'(0)=0$ occurs at only $\rho=0$ and satisfies $Q(0)=1/\mu^2$. Also
  \begin{equation*}
    \lim_{\rho\rightarrow\pm\infty}Q(\rho)=\frac{1}{(1-\mu)^2}.
  \end{equation*}
  Thus we have
  \begin{equation*}
    \frac{1}{\mu^2}\leq Q(\rho)<\frac{1}{(1-\mu)^2}\quad(1/2< \mu<1).
  \end{equation*}
  Hence
  \begin{equation*}
    {\rm Re}\{\phi(i\rho,\sigma)\}\leq 1-\frac{1}{2}\mu(1-\mu)\frac{1}{\mu^2}=\frac{3\mu-1}{2\mu}=\alpha
  \end{equation*}
  and this shows that $ {\rm Re}\{\phi(i\rho,\sigma)\}\not\in \Omega_\alpha$. Moreover, by Lemma \ref{MIMO} we get ${\rm Re}\{p(z)\}>0$ in $\Delta$, and concluding the proof.
\end{proof}
Finally, we give an estimate for the Jacobian of log--harmonic mappings of the form \eqref{repres}.
\begin{theorem}
  If $f(z)=z |z|^{2\beta} h(z)\overline{g(z)}$ is log--harmonic mapping, then
  \begin{equation*}
    \frac{(1-|\gamma|^2)(1-|z|^2)}{(1+|\gamma||z|)^2}|f_z|^2\leq J_f(z)\leq
  \left\{
  \begin{array}{ll}
    \frac{(1-|\gamma|^2)(1-|z|^2)}{(1-|\gamma||z|)^2}|f_z|^2 & \quad\hbox{$|z|<|\gamma|$,} \\ \\
     |f_z|^2 & \quad\hbox{$|z|\geq|\gamma|$,}
\end{array}%
\right.
  \end{equation*}
  where $\gamma=\overline{\beta}/(1+\beta)$, $\beta\in \mathbb{C}$, ${\rm Re}\{\beta\}>-1/2$ and $z\in \Delta$.
\end{theorem}
\begin{proof}
Let $\gamma=\overline{\beta}/(1+\beta)$, $\beta\in \mathbb{C}$ and ${\rm Re}\{\beta\}>-1/2$.
Consider the function
\begin{equation}\label{varphi}
  \varphi(z):=\frac{w(z)-\gamma}{1-\overline{\gamma}w(z)}\quad(z\in \Delta),
\end{equation}
where $w$ is the dilatation of $f$ of the form \eqref{equ. log}.
We have that $\varphi$ is analytic (because $|1+\overline{\beta}|>|\beta|$), $\varphi(0)=0$ and $|\varphi(z)|\leq 1$. Thus the function $\varphi$ satisfies the assumptions of Schwarz lemma which gives $|\varphi(z)|\leq |z|$ or
\begin{equation}\label{|w(z)-gamma|}
  |w(z)-\gamma|\leq |z||1-\overline{\gamma}w(z)|\quad(z\in \Delta).
\end{equation}
From \eqref{varphi}, we get
\begin{equation*}
w(z)=\frac{\varphi(z)+\gamma}{1+\overline{\gamma}\varphi(z)}.
\end{equation*}
This shows that the dilatation $w(z)$ is subordinate to
\begin{equation*}
\phi(z)=\frac{z(z+\gamma)}{1+\overline{\gamma}z}.
\end{equation*}
Since the linear transformation $\phi(z)$ maps $|z|=r$ onto the disc with the center
\begin{equation*}
  \left(\frac{x(1-|z|^2)}{1-|\gamma|^2 |z|^2},\frac{y(1-|z|^2)}{1-|\gamma|^2 |z|^2}\right)\quad(x={\rm Re}\{\gamma\}, y={\rm Im}\{\gamma\})
\end{equation*}
and the radius
\begin{equation*}
  \frac{(1-|\gamma|^2)|z|}{1-|\gamma|^2 |z|^2},
\end{equation*}
therefore, by use of the above and by the subordination principle, the inequality \eqref{|w(z)-gamma|} gives that
\begin{equation*}
  \left|w(z)-\frac{\gamma(1-|z|^2)}{1-|\gamma|^2|z|^2}\right|
  \leq\frac{(1-|\gamma|^2)|z|}{1-|\gamma|^2|z|^2}
\end{equation*}
and concluding the proof.
\end{proof}
\noindent
{\bf Acknowledgements.} The authors would like to thank the anonymous referee(s) for carefully reading of the manuscript and for the useful observations.

\end{document}